\documentclass[leqno,12pt]{article} 
\usepackage[dvipdfmx,colorlinks=true]{hyperref}
\hypersetup{urlcolor=blue, citecolor=red}
\usepackage{amsmath, amssymb, bm}
\usepackage{amsthm} 
\usepackage{amsmath,amssymb,mathtools}
\usepackage{mathrsfs} 
\usepackage{ascmac}
\usepackage{graphicx,color}
\theoremstyle{plain} 
\newtheorem{theorem}{Theorem}[section] %
\newtheorem{lemma}[theorem]{Lemma}

\newtheorem{proposition}[theorem]{Proposition}

\theoremstyle{definition} %

\newtheorem{remark}[theorem]{Remark}

\newcommand{\N}{\mathbb{N} }

\newcommand{\R}{\mathbb{R} }

\newcommand{\LC}{\left ( }
\newcommand{\RC}{\right ) }
\newcommand{\LD}{\left \{ }
\newcommand{\RD}{\right \} }

\newcommand{\DS}{\displaystyle }

\DeclareMathOperator*{\argmax}{argmax}

\DeclareMathOperator*{\conv}{conv}
\DeclareMathOperator*{\Proj}{Proj}

\makeatletter
\def\address#1#2{\begingroup
\noindent\parbox[t]{7.8cm}{%
\small{\scshape\ignorespaces#1}\par\vskip1ex
\noindent\small{\itshape E-mail address}%
\/: #2\par\vskip4ex}\hfill%
\endgroup}%
\makeatother
\title{\uppercase{Heat equation on the hypergraph containing vertices with given data}} %
\author{Takeshi Fukao, Masahiro Ikeda, 
Shun Uchida\footnote{Corresponding author: shunuchida@oita-u.ac.jp}
\date{} %
}

\begin{document}

\maketitle

%

\begin{abstract}
This paper is concerned with the Cauchy problem of a multivalued ordinary differential equation
governed by the hypergraph Laplacian,
which  describes the diffusion of ``heat'' or ``particles'' on the vertices of hypergraph.
We consider the case where the heat on several vertices 
are manipulated internally by  the observer,
namely,
are fixed by some given functions.
This situation can be reduced to a nonlinear evolution equation associated with a time-dependent subdifferential operator,
whose solvability has been investigated in numerous previous researches.
In this paper, however,  we give an alternative proof of the solvability 
in order to avoid some complicated calculations arising from the chain rule for the time-dependent subdifferential.
As for results which cannot be assured by the known abstract theory,
we also discuss the continuous dependence of solution on the given data
and the time-global behavior of solution.
\end{abstract}

\section{Introduction}

Let $V = \{ v _1 , \ldots , v_ n , v _ {n+ 1} , \ldots , v_{n+m} \} $ be a finite set,
$E \subset 2^{V }$ be a family of subsets consisting of two or more elements of $V$, and 
$w: E \to (0,\infty )$.
The triplet $G = (V ,E , w )$,
which is called {\it hypergraph},
can be interpreted as a model of a network structure 
in which vertices $v_ 1  ,\ldots , v_{n+m} \in V$ are connected by each hyperedge $e \in E$
(see Fig. \ref{Hypergraph}).

\begin{figure}[tp]
 \centering
 \includegraphics[keepaspectratio, scale=0.6]
      {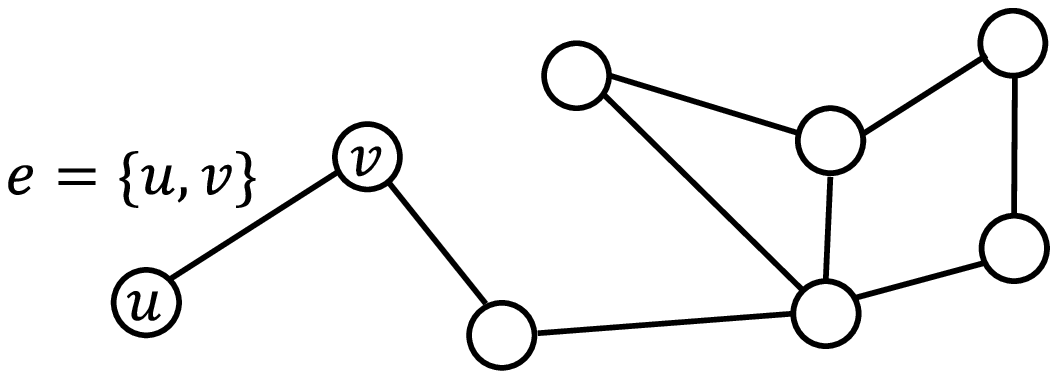}
      \hspace{4mm} 
 \includegraphics[keepaspectratio, scale=0.6]
      {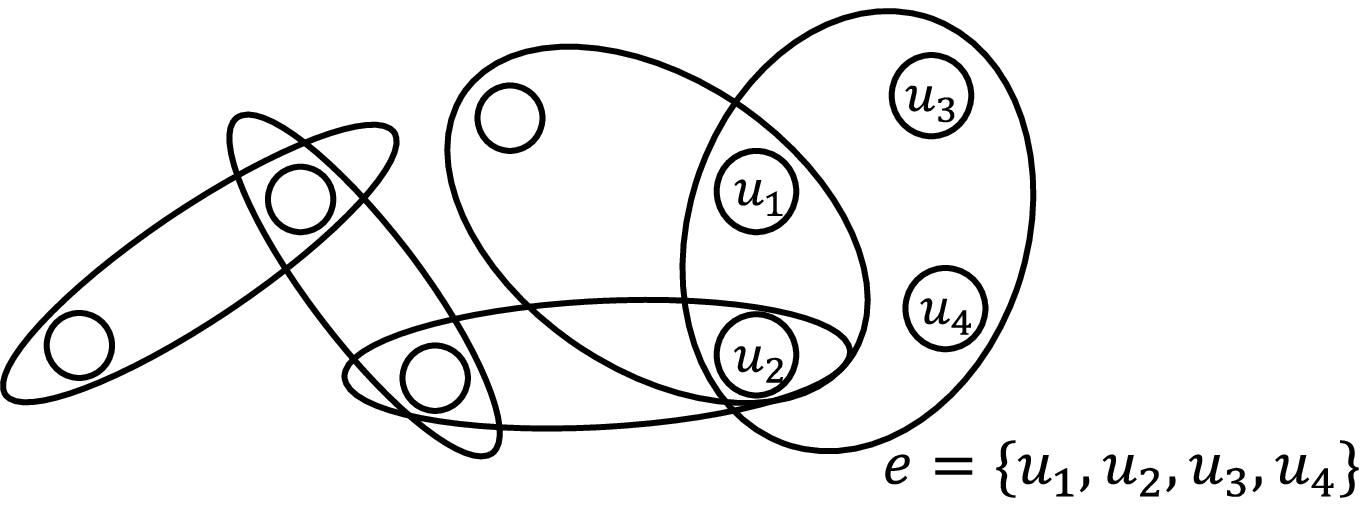}
 \caption[Usual Graph and Hypergraph]{When {\#}$e =2 $ for every $e\in E$,
each edge $e \in E $ can be regarded as a segment connecting two vertices (left figure, called {\it usual graph}).
Hypergraph is a generalization of usual graph 
which represents the connection and grouping of multiple members (right figure).
}
 \label{Hypergraph}
\end{figure}

In order to investigate the structure of network represented by a hypergraph,
Yoshida \cite{Yoshida00} introduced a nonlinear set-valued operator on $\R ^V $
(the family of mappings $x : V \to \R $)
as follows.
Define 
$ f_e (x) := \max _{u,v \in e} ( x (u ) - x(v))$
with respect to each $e \in E$
and 
\begin{equation}
\label{energy} 
\varphi _{G, p } (x)
:= \frac{1}{p} \sum_{e \in E} w (e) (f _e (x)) ^p ~~~~~p\in [1 ,\infty ) .
\end{equation}
These are convex functionals
with domains $D(f_e) = D(\varphi _{G, p }) = \R ^ V $
and then subdifferentiable at every $x \in \R ^V $.
By the chain rule
(see, e.g., \cite{CLT})
and the maximum rule 
(see, e.g., \cite[Proposition 2.54]{MN}),
the subdifferential of $\varphi _{G, p }$ coincides with 
\begin{equation}
\partial \varphi _{G, p} (x) 
= \LD \sum_{e \in E } w(e) (f_e (x) ) ^{p-1} b_e ;~b_e \in \argmax _{b \in B_e } b \cdot x  \RD ,
\label{Laplacian} 
\end{equation}
where $B _e \subset \R ^V $, 
called {\it base polytope} for $e \in E $, is defined by 
\begin{equation}
\label{BP} 
B_e := \conv  \{ 1 _u - 1 _ v ;~u,v \in e  \} 
\end{equation}
and $1 _ u : V \to \R $ with $ u \in V $ stands for 
\begin{equation}
\label{1u} 
1 _ u (v):
=
\begin{cases}
~~1 ~~&~~\text{ if } u =v , \\
~~0~~&~~\text{ if } u \neq v .
\end{cases}
\end{equation}
This operator $ L _{G ,p} = \partial \varphi _{G, p}  : \R ^V \to 2 ^{\R ^V } $
is called {\it hypergraph ($p$-)Laplacian}, where $1 \leq p < \infty $ (see also \cite{I-U}).

When $p =2$ and every $e \in E $ consists of two elements of $V$,
then $ L_{G , 2} $ becomes a linear operator on $\R ^V $,
i.e., a square matrix of order ${\#}V =  (n+m)$
and represents the movement of ``particles'' on a vertex to another adjacent vertex  through an edge $e\in E$ 
in the random walk model on the graph.
By analogy with this,
we can regard the ODE $x'(t) + L_{G, p} (x(t) ) \ni 0 $
with respect to $x : [0,T ] \to \R ^V $
as a diffusion model of ``heat'' or ``particles'' on each vertex $v_ i$, which is described by 
$x(t) (v _ i ) = x_ i (t)$.
In fact, 
the time-global behavior of solution to this ODE (studied in \cite{I-U})
quite resembles those to 
the PDE $\partial _t u - \nabla \cdot (| \nabla u | ^{p-2} \nabla u )  = 0 $.
In information science,
this ODE $x'(t) + L_{G, p} (x(t) ) \ni 0 $
is used to investigate the eigenvalues of  $L_{G, p}$,
prove the Cheeger like inequality, and 
analyze the cluster structure of network represented by hypergraph
(see,  e.g., \cite{FSY,IMTY,L-M,TMIY}).

Here let the heat on two vertices $v_ i , v_ j$
be fixed as $x (v _i ) = \alpha $ and $x (v _j ) = \beta   $ ($\alpha  > \beta $).
Under this situation,
the final state $x _{\infty } = \lim_{t \to \infty } x(t) $
of the diffusion process $x'(t) + L_{G, p} (x(t) ) \ni 0 $
might describe the thermal gradient in each path and 
then indicate the optimal path of heat flux from vertex $v _i $ to $v _j $.
As a generalization of this problem, 
we deal with the case where 
the heat at vertices $v_{n+1} , \ldots , v_{n+m }$
are determined by some given functions $a _j : [0,T ] \to \R $ ($j =1 ,\ldots ,m $),
i.e., the heat equation governed by $L_{G, p}$ under the condition
\begin{equation}
\label{constraint} 
x_{n+j} (t) = x (t) (v_{n+ j } ) = a _ j (t) ~~~~~j=1,\ldots ,m
\end{equation}
(see Fig. \ref{Givendata}).
This can be interpreted as a situation
where the observer internally manipulates the heat of the network.
As another example, the graph Laplacian \eqref{Laplacian}
and its energy functional \eqref{energy} on the usual graph
can be found in the definition of Laplacian on the fractal 
(see, e.g., \cite{K89,K01}). 
In this case, 
the assumption $a _ j \equiv 0$ 
corresponds to a homogeneous Dirichlet boundary condition on the fractal.

\begin{figure}[tp]
 \centering
 \includegraphics[keepaspectratio, scale=0.6]
      {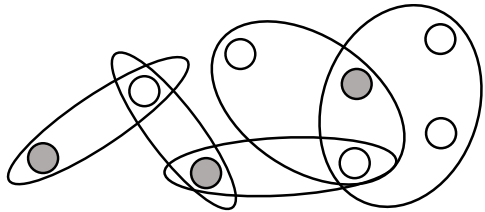}
 \caption[Given data]{We consider the case where 
the ``heat'' of some vertices (colored in the above) is given as $x _{n + i } (t) = a _i (t)$.
}
 \label{Givendata}
\end{figure}

We here state the setting of our problem more precisely.
Without loss of generality, we assume the hypergraph $G =(V ,E ,w)$ is {\it connected},
i.e.,  for every $u ,v \in V$ there exist some
$ u _ 1 , \ldots u _ {N-1} \in V $ and $ e _1 , e _2 , \ldots e_N \in E$ such that
$u_{j-1} , u _j \in e_j$ holds for any $j=1, 2, \ldots , N$, 
where $u_ 0 = u $ and $ u_N = v$
(if $G$ is disconnected,
we only have to divide $G$ into connected components and
consider the problem on each component, see \cite{I-U}).

By letting $ x_ i := x (v _ i ) $,
we can identify $\R ^ V $ (the family of mappings from $V$ onto $\R $)
with the Euclidean space $\R ^{n + m}$
(\eqref{1u} coincides with a unit vector of the canonical basis).
In this sense, $\R ^V $ is a Hilbert space 
endowed with the standard inner product 
$x \cdot y := \sum_{v \in V} x(v) y(v) $ and the norm $\| x \| := \sqrt{\sum_{v \in V } x(v) ^2 }  $.
For $x : [0,T ] \to \R ^V $, we shall write $x_ i(t):= x (t) (v_ i) $ for the same reason.

For convenience,
the given data in \eqref{constraint} will be unified by $a : [0 , T ] \to \R  ^V $ as
\begin{equation}
\label{a} 
a (t) := ( 0, \ldots , 0 , a_1 (t) , \ldots , a_ m (t)) . 
\end{equation}
Let  $K_a (t) \subset \R ^V $ be a constraint set which corresponds to the condition \eqref{constraint}
such that 
the former $n$ components $x_ 1 ,\ldots , x_n $ are chosen freely 
and the latter $m$ components $x_ {n+ 1} ,\ldots , x_{n + m } $ are fixed
by given data $a_ 1 (t ) , \ldots , a_m (t)$.
Namely, define
\begin{equation*}
K_a (t) := \{ x \in \R ^V ;~x= (x _1 , \ldots , x _n , a_1 (t) ,\ldots , a_m (t)  ), ~x_i \in \R ~(i =1,\ldots ,n ) \} .
\end{equation*}
Note that $K_a (t)$ is a closed convex subset in $\R ^V $
and the projection onto $K_a (t)$ can be defined by 
\begin{equation}
\label{proj}
\text{Proj} _{K_a (t) } x := (x_1 ,\ldots , x_ n , a _1 (t) , \ldots ,  a_ m (t)),
~~~\text{where ~}
x =(x_ 1 ,\ldots ,x _{n+m}) .
\end{equation}
Moreover,  let 
$I_{K_a (t)} : \R ^V \to [0 , \infty ]$ stand for the indicator on $K _a (t) \subset \R ^V $:
\begin{equation}
\label{Indi} 
I_{K_a (t)} (x)=\begin{cases}
~~0 ~~&~~\text{ if } x \in K_a  (t) , \\
~~ + \infty ~~&~~\text{ otherwise.}
\end{cases}
\end{equation}
We can define the subdifferential of $I_{K_a (t)} $ on $  K_a (t)$.
Here when $\psi : \R ^V \to ( - \infty , + \infty ]$
is a proper ($\psi \not \equiv + \infty $) lower semi-continuous and convex function,
its subgradient at $x \in \R ^V $ is defined by 
\begin{equation}
\label{sub}
\partial \psi (x) := \{ \eta \in \R ^V ;~ 
\eta \cdot (z - x ) \leq \psi (z) - \psi (x)
 ~~\forall z \in \R ^ V  \} .
\end{equation}
Then the diffusion process on the hypergraph with the condition \eqref{constraint} and the initial state $ x _ 0 \in \R ^ V $
can be represented by the following Cauchy problem of an evolution equation with constraint condition:
\begin{equation*}
\text{(P)}_{a , h , x _ 0 }
~~~~~~~
\begin{cases}
~~\DS x'(t) + \partial \varphi _{G, p } (x (t)) + \partial I_{K _a (t )} (x (t) ) \ni h (t) ~~~t \in [0, T],  \\
~~x(0) = x_ 0 ,
\end{cases}
\end{equation*}
where $h : [0,T ] \to \R ^ V $ is the given external force.

Since 
$\varphi _{G, p } (x) < \infty $ holds for every $x \in \R^ V $ (i.e., $D ( \varphi _{G, p }) = \R ^V $),
we have 
$\partial ( \varphi_{G, p }  + I_{K_a (t)} ) = \partial  \varphi_{G, p }  +\partial   I_{K_a (t)}$.
Therefore (P)$_{a , h , x _ 0 }$ can be reduced to a problem of evolution equation
governed by a subdifferential operator of time-dependent functional 
$ \varphi ^t : = \varphi_{G, p }  +  I_{K_a (t)} $,
which has been investigated in numerous papers, e.g., \cite{F-K,F-M-K,I-Y-K,Ken,S-I-Y-K}.
In this paper,
however,
we aim to introduce a simpler method for the constraint problem (P)$_{a , h , x _ 0 }$
by focusing on features of $\partial  \varphi _{G, p }$ and $\partial   I_{K _a (t)} $ stated in the next section.
We prove  the existence of solution to (P)$_{a , h , x _ 0 }$ in \S 3.
As for a result which cannot be assured by the known abstract theory,
we discuss the continuous dependence of solution on the given data $a (t)$ in \S 4.
Finally, we consider the time-global behavior of solution in \S 5
by using the facts given in previous parts.

\section{Preliminary}
In this section, we check and review some basic facts for later use. 
First, the subgradient of \eqref{Indi} can be written specifically as follows:
\begin{lemma}
\label{Ind-sub}
For any $t \in [0,T]$ and $x \in  K_a (t)$, 
the subdifferential of $I _{K_ a (t)}$ is characterized by 
\begin{equation*}
\partial I _{K_ a (t)} (x) =\{ (0 \ldots , 0 , \xi _{n+1} , \ldots , \xi _{n+m}) \in \R^ V , ~ \xi _ {n+j } \in \R ~ ~(j =1 ,\ldots , m) \} .
\end{equation*}
\end{lemma}
\begin{proof}
Let $\xi \in \partial I _{K_a (t)} (x)$. By the definition of subdifferntial \eqref{sub},
$0 \geq \xi \cdot (z - x)$
holds for every $z \in K _a  (t)$.
When  $x , z \in K _a (t)$, we have 
\begin{align*}
z - x &= (z _1 -x_1 , \ldots , z_ n - x_n , a_1 (t) - a_1 (t)  , \ldots , a_m (t) -a_m (t) ) \\
 &= (z _1 -x_1 , \ldots , z_ n - x_n , 0  , \ldots , 0 ) . 
\end{align*}
Hence it is necessary that
$0 \geq  \sum_{i=1}^{n} \xi _i  (z_i  - x _ i )$ holds 
for any $z_ i \in \R  $ ($i= 1, \ldots ,n $).
This implies that
$\xi _1 , \ldots ,  \xi_ n $ should be equal to zero  and 
$\xi _{n+1} ,\ldots , \xi _{n+m} $ can be chosen arbitrarily.
\end{proof}

We next check the following Poincar\'{e} type inequality
by repeating the argument in \cite{I-U}.
Here and henceforth, $\varphi _{G, p}$ in \eqref{energy} will be abbreviated to $\varphi $.
\begin{theorem}
\label{Poincare} 
Let the hypergraph $G =(V ,E ,w)$ be connected and $ x \in K _a (t)$.
Then there exists some constant $C$
depending only on $ p$ and $G $ such that 
\begin{equation}
\label{P-ineq} 
\sum_{k=1}^{n} |x_ k |  \leq C \varphi (x) ^{1/ p }+n \min _{1 \leq i \leq m} |a _ i (t) | .
\end{equation}
\end{theorem}

\begin{proof}

Let the index $j $ satisfy $|a _j (t)| = \min_{1\leq i\leq m} |a _i (t)|$.
Since $G$ is assumed to be connected,
for every fixed $v _k  $ ($k=1,\ldots ,n $)
there exist some 
$ u _ 1 , \ldots , u _ {N-1} \in V $ and $ e _1 , e _2 , \ldots , e_N \in E$ s.t.
$u_{l -1} , u _l  \in e_l$ ($l =1, 2, \ldots , N$) holds, 
where $u_ 0 = v _ k  $ and $ u_N = v_{n + j }$.
By $x(v _{n+j }) = a_ j (t)$ and H\"{o}lder's inequality,
we have  
\begin{align*}
|x (v _ k ) - a _ j (t) | 
&\leq \sum_{l =1}^{N} |x(u _ {l -1 } ) - x( u_{l} ) | \leq \sum_{l =1}^{N} f _{e _{l} } (x)
\leq \sum_{e \in E } f _{e  } (x) \\
&  \leq \frac{ ( {\#} V ) ^{1 / p' }}{ ( \min _{e \in E } w (e) ) ^{1/p}} \varphi (x )  ^{1/p} ,
\end{align*}
where $p' := p / (p-1 )$ is the H\"{o}lder conjugate exponent. 
Then we obtain \eqref{P-ineq}
with the positive constant $C =n (n+m) ^{1/ p' }  ( \min _{e \in E } w (e) ) ^{ - 1/p}$.
\end{proof}

Note that (see \cite[Theorem 2.4]{I-U})
\begin{equation}
\label{minimizer}
\varphi (y) = 0 ~~\Leftrightarrow ~~\exists c \in \R ~~\text{s.t.} ~~ y = (c , \ldots , c )  .  
\end{equation}
In addition, we can easily show that 
$f_e (x) = \max _{ u, v \in e} |x(u) - x(v)| \leq 2 \| x \| $
by definition
and  $ \| b \| \leq 2$ for any $b \in  B_e$
 since $1 _ u \in \R ^V $ can be identified with a unit vector of $\R ^{n+m }$.
Hence as for the boundedness of $\varphi $ and $\partial \varphi $,
\begin{equation}
\label{bound}
\varphi (x)\leq \frac{ 2 ^p  ({\#}E)   \max _{e\in E } w (e) }{p} \| x \| ^p , 
\hspace{5mm} 
\| \eta \| \leq 
2 ^p   ({\#}E)  \max _{e\in E } w (e)  \| x \| ^{p-1 } 
\end{equation}
hold for every $ x \in \R ^V$ and $\eta \in \partial \varphi (x)$, 
where ${\#}E$ is the number of hyperedges.

We also recall a variant of Gronwall's inequality as follows (see \cite[Lemme A.5]{Bre}).
Let
$L^q (0 ,T ; \R ^V )$ ($q \in [1, \infty )$)
and 
$L^{\infty } (0 ,T ; \R ^V )$
 denote the family of $g : [0,T] \to \R ^V$
which satisfies   
$ \int_{0}^{T} \| g(t) \| ^q  dt < \infty $ and 
$\sup _{0 \leq t \leq T }  \| g(t) \|  < \infty $, respectively.
Moreover, 
$g \in W ^{1, q} (0,T ; \R ^V )$ implies 
that $ g $ and its derivative $g'$ belong to $L^q (0 ,T ; \R ^V )$.
\begin{lemma}
\label{Gronwall} 
Let $\Xi  : [0, T ] \to \R $ be continuous 
and satisfy 
\begin{equation*}
\frac{1}{2} \Xi  (t) ^2 \leq \frac{1}{2} \kappa ^2 + \int_{0}^{t}   g(s) \Xi  (s) ds  ~~~\forall t \in [0,T]
\end{equation*}
with some  constant  $\kappa $ and non-negative function $g \in L^1 (0,T ; \R )$.
Then $\Xi  $ satisfies 
\begin{equation*}
|\Xi  (t) |  \leq | \kappa | + \int_{0}^{t}   g(s) ds  ~~~\forall t \in [0,T] .
\end{equation*}
\end{lemma}

\section{Existence of Solution}

We begin with the solvability of (P)$_{a , h , x_ 0 }$: 
\begin{theorem}
\label{Existence} 
Let 
$T \in (0, \infty  )$ and $a\in W ^{1,2} (0, T ; \R ^V )$.
Then for any $x_ 0 \in K_a (0)$ and $h\in L ^{2} (0, T ; \R ^V )$, 
\rm{(P)}$_{a , h , x _ 0 }$ possesses
 a unique solution $x \in W ^{1,2 } (0,T ; \R ^V )$. 
\end{theorem}

This result can be  assured
by known results for abstract theory of nonlinear 
evolution equations with time-dependent subdifferential. 
Indeed,
we can apply \cite[Theorem 1.1.2]{Ken}
since $ \varphi  ^t := \varphi +I _{ K _a (t) }$
satisfies the condition   (H)$_2$ in \cite[\S 1.5]{Ken} if $a \in W ^{1,2 } (0 ,T ; \R ^ V )$. 
In this section, however, we shall give an alternative proof 
which only relies on the basic properties of subdifferential 
in order to  avoid some complicated calculations    
arising from the chain rule of $ \frac{d}{dt}  \varphi  ^t  (x (t) )$. 
\begin{proof}
Remark that the Moreau-Yosida regularization of $ I_{K_a (t)}$ and its subdifferential
(the Yosida approximation of $ \partial  I_{K_a (t)} $)
coincide with 
\begin{align*}
( I_{K_a (t)})_{\lambda } (x) & :=
\inf _{y\in \R ^ V } \LD  \frac{1}{2\lambda } \| x - y \| ^2 +  I_{K_a (t)}(y) \RD 
=\frac{1}{2\lambda } \|x - \text{Proj} _{K_a (t)} x  \| ^2 , \\
 ( \partial  I_{K_a (t)} ) _{\lambda } (x) & :=   \partial  ( I_{K_a (t)} ) _{\lambda }  (x)
= \frac{1}{\lambda }  (x - \text{Proj}  _{K_a (t)} x )
\hspace{1cm} \lambda >0   .
\end{align*}
We first 
replace 
$ \partial I_{K_a (t)} $ in (P)$_{a , h , x_ 0 }$ with $ ( \partial  I_{K_a (t)} ) _{\lambda } $
and consider 
the following approximation problem:
\begin{equation*}
\text{(P)}_{\lambda }
~~~~
\begin{cases}
~~\DS x' _{\lambda }(t) + \partial \varphi (x_{\lambda }(t)) 
		+ \frac{x_{\lambda }(t) -\Proj _{K _a (t)} x_{\lambda }(t)}{\lambda } \ni h (t) , \\[3mm] 
~~x_{\lambda } (0) = x_0  .
\end{cases}
\end{equation*}
By the Lipschitz continuity of $\frac{1}{\lambda }  (x - \Proj _{K _a (t)} x ) $,
the abstract theory (see e.g, \cite[Th\'{e}or\`{e}me 3.6]{Bre})
yields the existence of a unique solution   $ x _{\lambda } \in W ^{1,2 } (0,T ; \R ^V )$ to (P)$_{\lambda }$
for every $ x_ 0 \in \R ^V$ and $ h \in L^2 (0,T ;\R ^V )$.

We establish a priori estimates 
independent of $\lambda $. 
Since 
\begin{equation*}
 x_{\lambda }(t) -\text{Proj} _{K _a (t)} x_{\lambda }(t)
  = (0 , \ldots , 0 , x_{\lambda , n+1} (t) -a_1 (t) ,\ldots , x_{\lambda , n+m} (t) -a_m (t)  ) ,
\end{equation*}
multiplying 
\begin{equation*}
( x _{\lambda }(t) -a (t) )' + \partial \varphi (x_{\lambda }(t)) 
		+ \frac{x_{\lambda }(t) -\Proj _{K _a (t)} x_{\lambda }(t)}{\lambda } \ni h (t) - a'(t)
\end{equation*}
by $x _{\lambda } - a $ and using \eqref{sub}, 
we have 
\begin{align*}
&  \frac{1}{2}  \frac{d}{dt} \|  x _{\lambda }(t) -a(t) \| ^2  + \varphi (x_{\lambda }(t)) - \varphi (a(t)) 
		+ \frac{1}{\lambda } \sum_{j=1}^{m} |x_{\lambda , n+j} (t) -a_j(t) |^2 \\
&\hspace{3cm} 			\leq \| h (t) -a ' (t) \|  \| x_{\lambda } (t) -a (t) \| .
\end{align*}
Hence if $x_{\lambda } (0) = x_0 =(x_{01} ,\ldots , x_{0n} , a_1(0) ,\ldots ,a_n (0)) \in K _a (0) $,
we obtain
\begin{equation}
\label{Exi002} 
\begin{split}
&\sup _{0\leq t \leq T} \|  x _{\lambda }(t) -a(t) \| ^2 
		+ \frac{2}{\lambda } \int_{0}^{T}  \sum_{j=1}^{m} |x_{\lambda , n+j} (t) -a_j(t) |^2 dt  \\
&		\leq 
\LC
 \sum_{i=1}^{n }  |  x _{0i } | ^2 +
2 \int_{0}^{T}  \varphi (a(t))  dt
+
\int_{0}^{T}  
\| h (t) -a' (t) \| ^2  dt \RC  \exp{T} .
\end{split}
\end{equation}
From this and 
\eqref{bound},
there exists some constant $C_1$ independent of $\lambda $ such that 
\begin{equation}
\label{Exi003} 
\begin{split}
\sup _{0\leq t \leq T} \| x _{\lambda } (t) \| \leq C_1 ,~~~
\sup _{0\leq t \leq T}\varphi (x _{\lambda } (t))
+ \sup _{0\leq t \leq T} \| \eta _{\lambda } (t) \|  \leq C_1 ,
\end{split}
\end{equation}
where $\eta _{\lambda } $ is the section of (P)$_{\lambda }$, i.e., fulfills  
\begin{equation*}
x' _{\lambda }(t) + \eta _{\lambda } (t) 
		+ \frac{1}{\lambda } (x_{\lambda }(t) -\text{Proj} _{K _a (t)} x_{\lambda }(t))= h (t) ,
		~~~~
		\eta _{\lambda } (t) \in \partial \varphi (x_{\lambda }(t))
\end{equation*}
for  a.e. $t \in [0,T ]$.
We next test (P)$_{\lambda } $ by $ (x_{\lambda }(t) - a(t)) '$.
Since 
$\varphi $ does not depend on the time, 
the standard chain rule $(\eta _{\lambda } (t) , x'_{\lambda }(t)) = \frac{d}{dt} \varphi (x _{\lambda } (t))  $
is valid (see \cite[Lemme 3.3]{Bre}).
Moreover, from 
 \begin{align*}
&(x_{\lambda } (t) -\text{Proj} 
_{K_a (t)} x_{\lambda } (t) , x'_{\lambda }(t) - a'(t)) \\
&=
\sum_{j=1}^{m} (x_{\lambda , n+j } (t) -a _j(t) ) (x_{\lambda , n+j } (t) -a _j(t) ) ' 
=
\frac{1}{2} \frac{d}{dt}  
\sum_{j=1}^{m} | x_{\lambda , n+j } (t) -a _j(t) | ^2 
\end{align*}
and \eqref{Exi003}, we can derive 
 \begin{align*}
& \frac{1}{2} \|  x ' _{\lambda }(t) -a ' (t) \| ^2 +\frac{d}{dt}  \varphi (x_{\lambda }(t)) 
		+
\frac{1}{2\lambda } \frac{d}{dt}  
\sum_{j=1}^{m} | x_{\lambda , n+j } (t) -a _j(t) | ^2  \\
&	\leq  C _1 \|  a' (t) \|  
+ \frac{1}{2} \| h (t) -a' (t) \| ^2  .
\end{align*}
By $x_{\lambda , n+j } (0) = a_{j} (0)$, we get 
\begin{equation}
\label{Exi004} 
\begin{split}
& \int_{0}^{T}   \|  x ' _{\lambda }(t) -a ' (t) \| ^2 dt
		+ 
\frac{1}{\lambda } \sup _{0 \leq t \leq T }  \sum_{j=1}^{m} | x_{\lambda , n+j } (t) -a _j(t) | ^2  \\
&\leq   
2 \varphi (x_0)  +2 C_1 \int_{0}^{T} \|  a' (t) \|  dt
+  \int_{0}^{T} \| h (t) -a' (t) \| ^2    dt.
\end{split}
\end{equation}

 By \eqref{Exi003} and \eqref{Exi004},
we can apply the Ascoli-Arzela theorem 
and extract  a subsequence of $\{ x_{\lambda } \} _{\lambda  > 0 }$
which  converges uniformly in $[ 0, T ]$
(we omit relabeling since the original sequence also  
converges as will be seen later).
Let $x \in C([0,T]; \R ^V )$ be its limit.
Furthermore,
 \eqref{Exi003} and \eqref{Exi004} also lead to
\begin{equation*}
\eta _{\lambda } \rightharpoonup \exists \eta , ~~~
x ' _{\lambda } \rightharpoonup x'  ~~~\text{weakly in } L^2(0,T; \R^ V )
\end{equation*}
and 
the demiclosedness of maximal monotone operators implies that the limit $\eta $ satisfies
$\eta (t) \in \partial \varphi (x(t))$ for a.e. $t \in (0,T )$.
Since \eqref{Exi004} yields 
\begin{equation*}
\sup _{0 \leq t \leq T }  \sum_{j=1}^{m} | x_{\lambda , n+j } (t) -a _j(t) | ^2  
	\leq   \lambda C_1 ,
\end{equation*}
we have $x_{ n+j } (t) = a _j(t)  $, i.e., $x (t) \in K _a (t)$ for every $t \in [0,T ]$.
By the equation, 
 \begin{equation*}
\frac{x_{\lambda } -\Proj _{K_a (\cdot )} x_{\lambda }}{\lambda } \rightharpoonup \exists \xi 
  ~~~\text{weakly in } L^2(0,T; \R^ V )
\end{equation*}
and 
$h(t) - x ' _{\lambda }(t) - \eta _{\lambda } (t) = \partial ( I_{K _a (t)})_{\lambda } (x _{\lambda }(t))$
(a.e. $t\in [0,T]$) hold.
Then it follows from \eqref{sub} that 
for every $v \in L^2(0,T ; \R ^V )$ satisfying $v(t) \in K _a (t)$,
\begin{align*}
&\int_{0}^{T}  (h(t) - x ' _{\lambda }(t) - \eta _{\lambda } (t)  , v(t) - x_{\lambda } (t)) dt \\
&\leq 
\int_{0}^{T}  ( I_{K_a(t)})_{\lambda } (v (t)) dt -  
\int_{0}^{T} ( I_{K _a (t)})_{\lambda } (x _{\lambda }(t)) dt \leq 0 .
\end{align*}
By taking its limit as $\lambda  \to 0 $,
we obtain 
\begin{equation*}
\int_{0}^{T}  (h(t) - x '(t) - \eta (t)  , v(t) - x(t)) dt \leq 0 
~~~~\forall v \in L^2(0,T ; \R ^V )~~\text{with}~~v(t) \in K _a (t) ,
\end{equation*}
which entails $\xi (t) = h(t) - x '(t) - \eta (t) \in \partial I _{K _a(t) } (x (t))$ a.e. $t \in [0,T ]$.
Thus the limit $x $ fulfills all requirements of the solution to (P)$_{a , h ,x_ 0 }$.

The uniqueness of solution can be guaranteed by
the monotonicity of $\partial \varphi $ and $\partial I_{K _a (t)}$.
Therefore
the limit $x$ is determined independently of the choice of subsequences
and then the original sequence $ \{ x_ \lambda  \}$ also converges to $ x $.
\end{proof}

\begin{remark}
We can find  several abstract results
for the case where the movement of the constraint set
is not smooth (see, e.g., {\cite[\S 4]{F-K}}).
However, it seems to be difficult to mitigate our assumption $a \in W ^{1 ,2 } (0 ,T ;\R ^V )$
since $ K _{a} (t)$ does not possess any interior point
and 
the smoothness of movement of $ K_ {a } (t)$
must exactly coincide with the regularity of the solution.
\end{remark}

\section{Continuous Dependence of Solutions on Given data}

We next consider the dependence of solutions with respect to the given data.
\begin{theorem}
\label{Dependence} 
Let
$(a ^k , h ^k  , x ^k _ 0 )
\in 
 W ^{1,2} (0, T ; \R ^V )
 \times 
 L ^{2} (0, T ; \R ^V )
 \times 
 K_{a^k } (0)$ 
 and let 
$x ^k \in W ^{1,2 } (0,T ; \R ^V )$ be the solution to 
\rm{(P)}$_{a ^k  , h ^k , x ^k _ 0 }$
($k=1,2$). 
 Then there exists a constant 
 $\gamma  $ depending only on $p$
and $G= (V, E,  w )$ such that 
 \begin{equation}
\label{Difference} 
\begin{split}
& \sup _{0\leq t \leq T } 
 \LC \sum_{i=1}^{n} |x^1_i (t) - x^2_i (t) | ^2  \RC ^{1/2}
\leq 
\LC \sum_{i=1}^{n} |x^1_{0i}  - x^2_{0i}  | ^2 \RC ^{1/2}  +
\int_{0}^{T}   \| h^1 (t ) -  h^2 (t ) \|   dt
  \\
& \hspace{2cm} +  \gamma 
\LC  
\sup _{0\leq t \leq T } \| x ^1 (t) \| ^{\frac{p-1}{2} }
+
\sup _{0\leq t \leq T } \| x ^2 (t) \| ^{\frac{p-1}{2} }
\RC
\LC  \int_{0}^{T}  \| a^1 (t ) -  a^2 (t ) \| dt \RC ^{1/2}.
\end{split}
\end{equation}
\end{theorem}

\begin{proof}
Let $\xi ^k \in \partial I _{K _{a^k} (\cdot  )} (x^k )$
and $\eta ^k \in  \partial \varphi (x^k ) $
be the sections of (P)$_{a ^k  , h ^k , x ^k _ 0 }$ ($k=1,2$),
namely, satisfy
\begin{equation*}
 \LC x ^k (t) \RC ' + \eta ^k (t) + \xi ^k (t) =h^k (t)
\end{equation*}
for a.e. $t \in [0,T]$.
Multiply the difference of equations, which is equivalent to
\begin{align*}
& \LC x^1 (t) -a ^1 (t)- x^2 (t) + a ^2 (t) \RC ' 
+ \eta ^1 (t)-  \eta ^2  (t)  
+\xi ^1 (t) - \xi ^2 (t)\\
&= h^1 (t) -h^2 (t)  - \LC   a ^1   (t) -  a ^2   (t) \RC ' ,
\end{align*}
by 
\begin{equation*}
x^1 (t) -a ^1 (t)- x^2 (t) + a ^2 (t)
= (x^1 _1 (t) - x^2 _1  (t) ,\ldots , x^1 _n (t) - x^2 _n  (t) , 0, \ldots , 0) 
\end{equation*}
(recall $ x^k (t) \in K_{a^k } (t)  $).
According to Lemma \ref{Ind-sub},
$ \xi ^k  (t) \in \partial I _{K _{a^k} (t )} (x^k (t))$
can be represented as $ \xi ^k (t) = ( 0,\ldots , 0 ,\xi ^k _{n+1} (t) ,\ldots , \xi ^k _{n+m} (t))$
and then fulfill
\begin{align*}
& \LC  \xi ^1  (t)  -\xi ^2  (t) \RC \cdot \LC x^1 (t) -a ^1 (t)- x^2 (t) + a ^2 (t) \RC  \\
&=
\sum_{i=1}^{n } 0 \cdot ( x^1 _i (t) - x^2 _i  (t) ) 
+
\sum_{j=1}^{m} (\xi ^1 _{n+j } (t) - \xi ^2_{n+j} (t) ) \cdot 0 = 0 .
\end{align*}
For the same reason, 
$ \LC  a ^1   (t) - a ^2   (t) \RC '  \cdot \LC  x^1 (t) -a ^1 (t)- x^2 (t) + a ^2 (t) \RC = 0$ 
also holds by 
$ ( a ^k(t) )' = (0, \ldots , 0 , ( a ^k _1 (t) )'  ,\ldots ,  ( a ^k _m (t) )' )$.
Therefore we can derive 
\begin{equation}
\label{Dif002} 
\begin{split}
 \frac{1}{2}
\frac{d}{dt}
\sum_{i=1}^{n}  | x^1 _ i (t)- x^2  _ i (t)  |^2 & \leq  
\LC \| \eta ^1 (t) \| + \| \eta ^2 (t) \| \RC  \|  a ^1 (t)-  a ^2 (t) \| \\
&\hspace{1cm}  + 
 \| h^1 (t) -h^2 (t)  \| 
	 \LC \sum_{i=1}^{n}  | x^1 _ i (t)- x^2  _ i (t)  |^2 \RC ^{1/2}  ,
\end{split}
\end{equation}
which provides \eqref{Difference}
by Lemma \ref{Gronwall}
with 
\begin{align*}
&\Xi  (t) = \LC \sum_{i=1}^{n} |x^1_i (t) - x^2_i (t) | ^2  \RC ^{1/2} , 
~~~~~
g(t) = \| h ^1  (t) - h^2 (t) \|  , \\
&\kappa = \LC \sum_{i=1}^{n} |x ^1 _ {0i}  - x ^2 _ {0i}  | ^2 
+ 2 \int_{0}^{T} \LC \| \eta ^1 (t) \| + \| \eta ^2 (t) \| \RC   \|  a ^1 (t)-  a ^2 (t) \| dt   \RC ^{1/2}  
\end{align*}
and \eqref{bound}.
\end{proof}

\begin{remark}
As shown in \cite[Remark 2.7]{I-U},
the hypergraph Laplacian is not strongly monotone,
i.e., it may be $ \partial \varphi (x ^1 ) =  \partial \varphi (x ^2 ) $ but $x ^1 \neq x^2 $.
Hence it is not easy to deduce
better information from $ ( \eta ^1 (t) - \eta ^2 (t) ) \cdot ( x ^1 (t) - x ^2 (t) ) $
in \eqref{Dif002} than above in general.
\end{remark}

\begin{remark}
Theorem \ref{Dependence}  implies that
the solution to (P)$_{a   , h , x _ 0 }$ is 
$1/2$-H\"{o}lder continuous 
with respect to the given data $a (t)$.
This seems to be optimal and it is difficult to derive better estimate in general.
Indeed, let $E = \{ V \}$, 
$x ^k (t) = ( x^k _1 (t) ,\ldots ,  x^k _n (t) ,a^k _1 (t) ,\ldots ,  a^k _m (t) ) $
be a solution to \rm{(P)}$_{a ^k  , h ^k , x ^k _ 0 }$ ($k=1,2$),
and
$\eta ^k (t) \in \partial \varphi (x ^k (t)) $ be the section of
\rm{(P)}$_{a ^k  , h ^k , x ^k _ 0 }$.
Moreover, we assume the initial data for $x ^1 $ satisfies 
\begin{equation*}
x^1 _{i_1} (0) < a ^1 _j (0) < x^1 _{i_2} (0) ~~~~\forall j = 1, \ldots , m 
\end{equation*}
with some $ i _1, i_2 \in \{ 1 ,\ldots , n \}$
and  
for 
$x ^2 $ fulfills 
\begin{equation*}
a^2 _{j_1} (0) <   x^2 _{i} (0)  < a^2 _{j_2} (0) ~~~~\forall i = 1, \ldots , n 
\end{equation*}
with some $ j _1, j_2 \in \{ 1 ,\ldots , m \}$.
By the continuity of solutions and $a^k$,
there exists some $T _ 0 > 0$ such that the solutions still satisfy
$x^1 _{i_1} (t) < a ^1 _j (t) < x^1 _{i_2} (t)$ and $a^2 _{j_1} (t) <   x^2 _{i} (t)  < a^2 _{j_2} (t)$ 
($\forall i = 1, \ldots , n$, $\forall j = 1, \ldots , m $)
on $ [0 ,T _ 0 ]$.
Then
we can see by 
\eqref{Laplacian} that
$\eta ^k (t) \in \partial \varphi (x ^k (t) )$ with $ t \in [0 ,T _ 0 ]$
can be represented by
\begin{equation*}
\eta ^1 (t) = ( \eta ^1 _{ 1}(t) ,\ldots , \eta ^1 _{n}(t)  , 0 , \ldots ,  0 ), ~~~~
\eta ^2 (t) = (0 , \ldots ,  0 , \eta ^2 _{n+ 1}(t) ,\ldots , \eta ^2 _{n+m }(t) ).
\end{equation*}
Hence we can obtain 
\begin{align*}
&(\eta ^1 (t)- \eta ^2  (t)) \cdot ( x^1 (t)- a ^1 (t) - x^2 (t) + a^2 (t)) \\
&=
\eta ^1 (t) \cdot (  x^1 (t)- a ^1 (t) - x^2 (t) + a^2 (t)) 
=
\sum_{i=1}^{n}
 \eta ^1 _ i (t) (  x^1 _ i (t)  - x^2 _ i  (t)  ) ,
\end{align*}
which appears in the estimate above.
It is not obvious whether we can deduce the term $ \| x^1 (t)   - x^2 (t) \| ^\alpha  $ $(\alpha > 1)$
from this.
\end{remark}

\section{Asymptotic Behavior of Global Solution}

Throughout this section, let $x (t) $ stand for a unique solution to (P)$_{a , h ,x _0}$.
According to  Theorem \ref{Existence},
the solution can be globally extended if 
\begin{equation}
\label{A1}
a \in W ^{1,2} _{\text{loc}} (0, \infty  ;\R ^ V) ,~~~~
h \in L ^{2} _{\text{loc}} (0, \infty  ;\R ^ V) .
\end{equation}
We discuss the global behavior of solution.
In addition to the above, we assume 
\begin{equation}
\label{A2}
a' , h' \in L ^1 (0, \infty ; \R ^V ) 
\end{equation}
so that 
$ a _{\infty } = \lim _{ t \to \infty } a(t) $
and
$ h _{\infty } = \lim _{ t \to \infty } h(t) $ 
can be determined uniquely.

We first consider the case where 
$a$ is independent of the time.
Henceforth, 
$ K _a (t ) $ is abbreviated to $K _a$ when $a$ is a constant vector.
\begin{proposition}
\label{GB01} 
Let $a = (0 ,\ldots , 0 , a_{1} , \ldots , a_{m}) \in \R ^V $ be a constant vector
and $h$ satisfy \eqref{A1} and \eqref{A2}.
If $h _{\infty } \in R(\partial \varphi + \partial I _{K_a } ) $ and 
$h - h_\infty \in L ^1(0,\infty  ; \R ^V )$,
then 
$x_{\infty} = \lim _{t \to \infty } x(t)$ is determined uniquely 
and satisfies
$\partial \varphi (x_{\infty }) + \partial I _{K_a } (x_{\infty }) \ni h_{\infty }$,
that is, 
$x_{\infty }$ is a stable solution to (P)$_{a , h_{\infty} ,x _{\infty}}$. 
\end{proposition}
\begin{proof}
We can apply \cite[Th\'{e}or\`{e}me 3.11]{Bre},
since
$\partial ( \varphi +  I_{K_{a}} )= \partial \varphi + \partial I_{K_{a}}$ is time-independent.
\end{proof}

\begin{remark}
Let $a = (0 ,\ldots , 0 , a_{1} , \ldots , a_{m}) $ be a constant vector.
By the definition of subdifferential \eqref{sub},
$h_{\infty } \in R( \partial \varphi  + \partial  I _{ K _{a }})$
holds if and only if
the closed convex set 
$ \{ x \in \R ^V ;~\Phi (x) = \min _{y \in \R ^V} \Phi (y) \} $ is not empty,
where
$\Phi : \R ^V \to \R  $
is a proper lower semi-continuous convex function defined by 
\begin{equation*}
\Phi (x) := 
\begin{cases}
~~\DS \varphi (x) + I _{ K _{a }} (x) - h _{\infty} \cdot x ~~&~~ \text{ if } x \in K_{a } ,\\
~~ \DS + \infty ~~&~~ \text{ if } x \not \in K_{a } .
\end{cases}
\end{equation*}
From this and $ \varphi  , I _{ K _{a }} \geq 0 $,
we can derive $0 \in R( \partial \varphi  + \partial  I _{ K _{a }})$ for every $p \geq 1 $.
Since Theorem \ref{Poincare} 
 implies 
 \begin{equation*}
 \Phi (x) \geq \frac{1}{2 ^p C ^ p }  \| x  \| ^p - \frac{(n+1) ^p}{C ^p } \LC \sum_{j=1}^{m} |a _ { j}|  \RC ^p
 - h_{\infty } \cdot x ,
\end{equation*}
there exists a global minimizer for every $h _{\infty }$, i.e.,
$R( \partial \varphi  + \partial  I _{ K _{a }}) = \R ^V $ holds if $p > 1 $.
When $p=1 $, however, we can see that $ R( \partial \varphi  + \partial  I _{ K _{a }}) \neq \R ^V $.
Indeed, 
\begin{equation*}
a  \equiv  0,~~~ h_{\infty } = (4 ({\#}E) \max _{e\in E} w(e) , 0 , \ldots , 0 ), ~~~  
x ^ \mu  = ( \mu , 0 , \ldots , 0 ) \in K_{a }~~~ (\mu = 1 , 2, \ldots ) 
\end{equation*}
satisfy
$\Phi (x _{\mu }) \to - \infty $ as $\mu \to \infty $ (see \eqref{bound}).
\end{remark}

We next deal with the case where $a $ depends on $t$
and $ h (t) , a(t) \to  0 $ as $t \to \infty $.
\begin{theorem}
\label{ABTh01}
Let \eqref{A1}, \eqref{A2}, and
$ h \in L^1 (0 , \infty ; \R ^V ) \cap L^2 (0 , \infty ; \R ^V ) $ be assumed.
If $a \in L^p (0, \infty ; \R ^V )$ and $ a' \in  L ^2 ( 0 , \infty ;\R ^V )$ are satisfied, 
then $x(t) \to 0 $ as $t \to \infty $. 
In particular, 
if $ h = a  \equiv 0$, then 
there exists some constant $\gamma > 0 $ depending only on $p$ and $G = (V ,E ,w)$ such that 
\begin{equation}
\label{decay} 
\| x(t) \| \leq 
\begin{cases}
~~\DS \LC  \| x_ 0 \| ^ {2 -p } - \gamma t \RC ^{ \frac{1}{2-p} } _+ ~~
			&~~\text{ if } 1 \leq p <2 ,\\
~~\DS   \| x_ 0 \|  \exp \LC -\gamma t \RC ~~
			&~~\text{ if } p =2 ,\\
~~\DS \LC  \| x_ 0 \| ^ {- (p-2 ) } +  \gamma t \RC  ^{-  \frac{1}{p-2} } ~~
			&~~\text{ if }  p >2 ,
\end{cases}
\end{equation}
where $(s) _+ := \max \{  0, s \}$ stands for the positive part of $s$. 
\end{theorem}

\begin{proof}
Henceforth, let $c_0 , C_ 0$ denote general constants independent of $T $. 
Note that the assumptions
lead to $a , h  \in L^{\infty } (0, \infty ; \R ^ V)$ and
$ a _{\infty } = h _{\infty} = 0$. 
%
%

Testing (P)$_{a , h , x_0 }$, which is equivalent to
\begin{equation*}
 ( x (t) -a  (t) )'  + \partial \varphi (x(t)) + \partial I_{K_ {a} (t)} ( x(t) ) \ni h (t) - a'(t)  ,
\end{equation*}
by $(x -a )=
(x _1 (t) , \ldots , x  _n  (t) , 0 , \ldots ,0)$ (recall \eqref{a}), we obtain
\begin{equation}
\label{AB000} 
\frac{1}{2}  \frac{d}{dt} \sum_{i=1}^{n} |x_ i(t)| ^2  +\varphi (x(t) ) -\varphi (a (t) )  
\leq   \|  h  (t) \|  \LC  \sum_{i=1}^{n} |x_ i(t)| ^2 \RC ^{1/2}  .
\end{equation}
Here we use the fact that
$\xi  (t) \cdot ( x (t) -a  (t)  ) = 0$ holds for every $\xi (t) 
\in \partial I _{K _{a} (t)} (x (t))$ by Lemma \ref{Ind-sub}
and 
$  a ' (t)  \cdot ( x (t) -a (t)  ) = 0$ by 
$a '(t)  = (0, \ldots , 0 , a' _1 (t)  ,\ldots , a ' _m (t)  )$.
By  
Lemma \ref{Gronwall}, 
we have for arbitrary $T > 0 $ 
\begin{equation}
\label{Dif001}
\sup _{ 0 \leq t \leq T } \LC \sum_{i=1}^{n} |x_i (t)| ^2 \RC ^{1/2}
\leq   
\LC \sum_{i=1}^{n}|x  _{0i}|^2 
+
 2 \int_{0}^{T}  \varphi (a  (t) ) dt \RC ^{1/2} 
+  \int_{0}^{T}  \| h   (t)  \| dt .
\end{equation}
Since $a \in L ^p (0, \infty ; \R ^V )$ is assumed and 
$\varphi (a (t)) \leq C_ 0  \| a (t) \| ^p $ holds by \eqref{bound}, 
it follows that 
$ \sup _{ 0 \leq t < \infty  }\| x (t) \|  $ is bounded.
This together with \eqref{bound} implies that
the section $\eta  \in \partial \varphi (x )$ of (P)$_{a , h , x _ 0 }$ satisfies
$\sup _{ 0 \leq t < \infty  }\| \eta  (t) \| \leq  C_0 $. 
Moreover,  Theorem \ref{Poincare} yields  
\begin{equation*}
 \varphi (x(t) )    
 \geq \frac{1}{2^p C}  \LC  \sum_{i=1}^{n} |x _ i (t)|  \RC ^{p}     
 -
 \frac{n ^p }{2^p C} \LC  \sum_{j=1}^{m} |a _ j (t)|  \RC ^{p}   .
\end{equation*}
Hence 
integrating \eqref{AB000} over $[0, \infty ) $, we also get 
$  \int_{0}^{\infty } \| x(t) \| ^p  dt  \leq  C_0 $.

Next multiplying  (P)$_{a , h , x _ 0 }$ by $(x - a ) ' =  (x' _1 (t  ) , \ldots ,x '_n (t) ,0 , \ldots ,0  )$
and using the boundedness of $ \eta (t)$,
we have 
\begin{equation*}
\frac{1}{2} 
\sum_{i=1}^{n} | x' _i (t) | ^2 + \frac{d}{dt} \varphi (x (t ) ) 
 \leq C_0 \|  a ' (t) \|
 + \frac{1}{2} \| h (t ) \| ^2  .
\end{equation*}
We here used Lemma \ref{Ind-sub} and $ a ' (t) \cdot ( x (t) -a  (t)  )' = 0$.
From $a ' \in L^1 (0 , \infty ;\R^ V )  $ and $h \in L^2 (0 , \infty ;\R^ V ) $,
we can derive
$\sum_{i=1}^{n}  \int_{0}^{\infty } | x' _ i (t) | ^2 dt \leq C_0  $.  

Define the {\it $\omega $-limit set } of the solution $x$ to (P)$_{a , h ,x_ 0}$ by
\begin{equation*}
\omega (a , h ,x_ 0  ) := \{ y \in \R ^V ;~ \exists \{ t _ k \} _{k\in \N } ~~\text{ s.t. }~~
t _k \to \infty ~~x(t_ k) \to y ~~\text{ as } k\to \infty  \} .
\end{equation*}
By the global boundedness of $ x(t)$ given above,
$\omega (a , h ,x_ 0  )  \neq \varnothing $  holds under the assumptions of Theorem \ref{ABTh01}. 
Let $x _{\infty } \in \omega (a , h ,x_ 0  )$ and 
$ \{ t _ k \} _{k\in \N }$ satisfy 
$x (t _k ) \to x _{\infty }$ and  $t _k  \to \infty $ as $k\to \infty $. 
Note that $x _{\infty } \in K _{a _{\infty } }  $
since $ a (t ) \to a _{\infty }$ and $ x _{n + j} (t) = a _j (t)$.

From $x'  \in L^2 (0, \infty ; \R ^V )$, 
it follows that for every $s \in [0,1 ]$
\begin{equation*}
\LC \sum_{i=1}^{n} | x _i (t _ k +s ) - x_ i  (t _ k )  | ^2 \RC ^{1/2}
 \leq  \LC  \sum_{i=1}^{n} \int_{t_k }^{t_k +1 } | x ' _i (\tau ) | ^2 d \tau  \RC ^{1/2} 
 \to 0 
\end{equation*}
as $t _ k \to \infty $,
that is, $ x _i (t _ k +s ) \to x _{\infty i}$ holds uniformly with respect to $s \in [0,1 ]$.

We here define $X _k \in W ^{1,2} ( 0,1 ; \R ^V ) $ by $X _k (s) := x (t _k + s )$.
Then it is easy to see that
$X _k (s) \to x_ {\infty }$ uniformly in $[0,1 ]$
and 
$X _k  $ is a solution to 
\begin{equation*}
 X ' _k (s) + \partial \varphi (X_k (s)) + \partial I _{K _{a }(t_ k +s )} (X_k (s) ) \ni h (t _ k +s ) .
\end{equation*}
Let $ \zeta  _k (s) \in \partial \varphi (X_k (s)) + \partial I _{K _{a } (t_ k +s )} (X_k (s) ) $
be the section of this inclusion, 
then
\begin{align*}
\LC \int_{0}^{1} \| \zeta  _k (s) \| ^2 ds \RC ^{1/2 }
& \leq   
\LC \int_{0}^{1} \|  X ' _k (s) \| ^2 ds \RC ^{1/2 } 
+
\LC \int_{0 }^{1}  \|  h  (t _ k +s) \| ^2 ds \RC ^{1/2 } \\
&
=
\LC \int_{t _ k }^{t _ k + 1} \|  x'  (s) \| ^2 ds \RC ^{1/2 } 
+
\LC \int_{t _ k }^{t _ k + 1}  \|  h  (s) \| ^2 ds \RC ^{1/2 } .
\end{align*}
Hence 
$\zeta  _ k \to 0$ in $ L ^2 (0,1 ;\R ^V )$
by  $ a' , h \in L^2 (0, \infty ; \R ^V )$.

Choose $y_1 , \ldots , y _ n \in \R $ arbitrarily
and define 
\begin{equation*}
Y _ k (s) := (y_1 , \ldots , y _ n , a _1 (t_k +s) ,  \ldots , a _m (t_k +s) ) \in K _{a } (t _ k +s ) .
\end{equation*}
Obviously, $Y _ k (s) $ converges to $y _{\infty } 
:= (y_1 , \ldots , y _ n ,$ $ a _{\infty 1 } , \ldots , a _{\infty m } ) \in K _{a _{\infty }}$
as $k \to \infty $
uniformly with respect to  $s \in [0, 1]$.
By the definition of subdifferential \eqref{sub}
and $   \partial ( \varphi  + I _{K _{a} (t_ k +s )} ) =  \partial \varphi  + \partial I _{K _{a }(t_ k +s )} $,
we obtain
\begin{equation*}
\int_{0}^{1 } \zeta _k (s) \cdot (Y_ k (s) - X _k (s) ) ds \leq 
\int_{0}^{1 } \LC  \varphi ( Y_ k (s) ) - \varphi ( X _k (s) ) \RC ds   .
\end{equation*}
Since $a , x \in L^{\infty } ( 0 ,\infty ; \R ^V )$ and  \eqref{bound},
the dominated convergence theorem is applicable to this and  
\begin{equation*}
0 \leq \int_{0}^{1 } \LC  \varphi ( y _{\infty } ) - \varphi ( x _{\infty }  ) \RC ds
=    
 \LC  \varphi ( y _{\infty } ) +I _{ K _{a _ \infty}} (y _{\infty } ) \RC 
- \LC  \varphi ( x _{\infty } ) +I _{ K _{a _ \infty}} (x _{\infty } ) \RC
\end{equation*}
holds for every $y_1 , \ldots , y _ n \in \R $, which
implies that $\varphi + I _{ K _{a _ \infty}} $ attains its minimum at $x _{\infty }$. 
When $a _{\infty }= 0$,
the minimum of $\varphi + I _{ K _{a _ \infty}} $ coincides with $0$.
From this fact and \eqref{minimizer}, we can derive $x _{\infty } =  ( 0 , \ldots , 0 )$.
Since the limit is determined independently of the choice of subsequences $\{ t_ k \}$,
the whole trajectory $\{ x(t) \}$ also tends to $0$.
Therefore we can conclude that $\omega (a, 0,x_0 ) = \{ 0 \}$.

Especially, if $a, h  \equiv 0$,  then 
\eqref{AB000} becomes 
\begin{equation*}
\frac{1}{2} \frac{d}{dt} \sum_{i=1}^{n} |x_ i(t)| ^2 
+ \varphi (x (t )) \leq \varphi (0 )  = 0 .    
\end{equation*}
Furthermore, \eqref{P-ineq} leads to 
$C  \varphi (x(t) )   
 \geq  \LC  \sum_{i=1}^{n} |x _ i (t)|  \RC ^{p}$.
Hence multiplying (P)$_{0, 0 , x _ 0 }$ by $x = (x-a )$, we get 
\begin{equation*}
\frac{1}{2}  \frac{d}{dt} \| x (t) \| ^2  +  c_ 0  \| x(t) \| ^p    
\leq  0
\end{equation*}
 which immediately yields \eqref{decay}. 
\end{proof}

We next consider that case where $  a _{\infty } \neq 0$.
When $\int_{0}^{\infty } \varphi (a (t)) dt   < \infty $,
there exists some subsequence $\{ t_k \}  _{k\in \N }$  
such that $t _k \to \infty$ and $\varphi ( a (t _k ) ) \to \varphi ( a _{\infty } ) = 0   $,
which together with \eqref{minimizer} leads to $a _{\infty } = 0$.
This implies that
$\sup _{ 0\leq t < \infty  } \| x (t) \| < \infty $ can not be deduced 
via a priori estimate \eqref{Dif001}.
We here give another way to establish uniform estimates by using Theorem \ref{Dependence}. 
\begin{theorem}
\label{ABTh03} 
Let $1 \leq p <3$,
 \eqref{A1}, and \eqref{A2} be satisfied.
Moreover, suppose that
\begin{align*}
&a '  \in L^2 (0, \infty ; \R ^V ), ~~~
h _ {\infty } \in R (\partial \varphi + \partial I _{K _{a _{\infty }}}), \\
&a - a_{\infty} \in  L^1 (0 , \infty ;\R ^V ) , ~~~
h - h_{\infty} \in  L^1 (0 , \infty ;\R ^V ) \cap  L^2 (0 , \infty ;\R ^V ).
\end{align*}
Then 
$x_{\infty} = \lim _{t \to \infty } x(t)$ is determined uniquely 
and satisfies
$\partial  \varphi (x _{\infty } ) + \partial I _{K_{a _{\infty}}} (x _{\infty } ) \ni h _{\infty }$.
\end{theorem}

\begin{proof}
Let $z _{\infty } \in \R ^V $ satisfy
$h _{\infty } \in \partial \varphi (z _{\infty } ) + \partial I _{ K_{a _ \infty }} (z _{\infty } )$,
which means that 
$ z _{\infty }$ is a solution to (P)$_{a _ {\infty} , h _{\infty} , z _{\infty }} $.
By Theorem \ref{Dependence} 
with $ (a ^2 , h ^2 , x ^2 _ 0 ) = (a _ {\infty} , h _{\infty } ,z _{\infty })$, 
we obtain 
 \begin{equation*}
\sup _{0\leq t \leq T } 
 \LC \sum_{i=1}^{n} |x_i (t) - z _{\infty i } | ^2  \RC ^{1/2}
\leq 
C_0 
\LC  
\sup _{0\leq t \leq T } \| x  (t) \| ^{\frac{p-1}{2} } +
1  
\RC ,
\end{equation*}
where $C_ 0 $ is a general constant independent of $T > 0 $.
%
%
%
Then when $p=1$,
we can obtain $ \sup _{0 \leq t < \infty } \| x(t )\| < \infty  $.
Otherwise, 
let $Z (T) := \sup _{0 \leq t \leq T} 
\LC \sum_{i=1}^{n} | x _i (t) - z _{\infty  i }| ^2  \RC ^{1/2 }  $,
which satisfies $Z(T) \leq C_0 (Z (T ) ^{ \frac{p-1}{2}  } +1 )$ from above.
Hence 
if $1<  p < 3$,
we obtain 
$Z(T) \leq \max \{ (C_0 + 1) ^{2 / (3-p  )} , C _ 0 ^{2/ (p-1 ) } \}$.
By the arbitrariness of $T$,
it follows that 
$ \sup _{0 \leq t < \infty } \| x(t )\| < \infty  $,
which leads to 
$ \sup _{0 \leq t < \infty } \| \eta (t ) \| < \infty  $.
Testing 
(P)$_{a , h ,x_ 0 } $ by $(x -a ) '$, we have 
\begin{align*}
&\| x ' (t) -  a' (t) \| ^2 + \frac{d}{dt} \varphi (x(t) ) \\
& \leq
\| \eta (t ) \| \| a' (t) \|  
+\| h (t) - h_{\infty} \| \| x ' (t) -  a' (t) \| + h _{\infty } \cdot  ( x (t) -  a(t) )'
\end{align*}
Here $ \int_{0}^{T} h _{\infty } \cdot  ( x (t) -  a(t) )' dt 
=
  h _{\infty } \cdot  ( x (T) -  a(T) - x_ 0 + a(0) ) 
$
is uniformly bounded by \eqref{A2} and $\sup _{ 0\leq t < \infty } \| x (t) \| < \infty $.
Hence $\int_{0}^{\infty } \| x ' (t) \| ^2 dt  < \infty $ also holds.

Fix $x_{\infty } \in \omega (a , h ,x_ 0)$ 
and let $ \{ t _ k \} _{k\in \N }$ satisfy $t _k  \to \infty $ and 
$x (t _k ) \to x _{\infty }$ as $k\to \infty $.
As above, define
$X _k  := x (t _k + \cdot ) \in W ^{1,2} ( 0,1 ; \R ^V ) $,
which is the solution to 
\begin{equation*}
 X ' _k (s) + \partial \varphi (X_k (s)) + \partial I _{K _{a }(t_ k +s )} (X_k (s) ) 
- h_{\infty }\ni h(t_k + s) - h_{\infty } . 
\end{equation*}
Then the section $ \zeta  _k (s) \in \partial \varphi (X_k (s)) + \partial I _{K _{a }(t_ k +s )} (X_k (s) ) $
of this inclusion satisfies
\begin{equation*}
\LC \int_{0}^{1} \| \zeta  _k (s) - h_{\infty} \| ^2 ds \RC ^{1/2 } = 
\LC \int_{t _ k }^{t _ k + 1} \|  x'  (s) -  h   (s) + h_{\infty } \| ^2 ds \RC ^{1/2 } 
\to 0 ,
\end{equation*}
namely, $\zeta  _ k \to h _{\infty }$ in $ L ^2 (0,1 ;\R ^V )$.
For every $y_1 , \ldots , y _ n \in \R  $ we can obtain
\begin{equation*}
h_{\infty } \cdot (y_{\infty} - x _{\infty}) \leq
 \LC  \varphi ( y _{\infty } ) +I _{ K _{a _ \infty}} (y _{\infty } ) \RC 
- \LC  \varphi ( x _{\infty } ) + I _{ K _{a _ \infty}} (x _{\infty } ) \RC ,
\end{equation*}
where
$y _{\infty } = (y_1 , \ldots , y _ n , a_{\infty  1} , \ldots , a_{\infty  m} )$.
Therefore $y \mapsto \varphi (y) + I _{ K _{a _ \infty}} (y) - h _{\infty } \cdot y$
attains its minimum at $x _{\infty }$,
 which implies that
 $\partial \varphi (x_{\infty}) + \partial I_{K_{a _{\infty}}} (x_{\infty } ) \ni h _{\infty }$.

Again, choose  $x_{\infty } \in \omega (a , h ,x_ 0)$ arbitrarily  
and let $ \{ t _ k \} _{k\in \N }$ satisfy $x (t _k ) \to x _{\infty }$.
Since $x_{\infty }$ is a solution to  
 (P)$_{a _ {\infty} , h _{\infty} , x _{\infty }} $,
we can derive from Theorem \ref{Dependence} with $t = 0 $ replaced by $t = t_ k $
\begin{align*}
&\LC \sum_{i=1}^{n} |x_{i} (t + t_k ) - x_{\infty i} | ^2  \RC ^{1/2}
\leq \LC \sum_{i=1}^{n} |x_{i} ( t_k ) - x_{\infty i} | ^2  \RC ^{1/2} \\
&\hspace{3cm} + C _0 \LC  \int_{t_ k }^{\infty } \| a (t) - a _{\infty } \| dt  \RC ^{1/2} 
+ \int_{t_ k }^{\infty } \| h (t) - h _{\infty } \| dt 
\end{align*}
for every $t > 0 $.
Hence   
$x(t) \to x_{\infty} $  
can be assured regardless of the choice of subsequences, that is,
the limit $x_{\infty }$ is determined uniquely.
\end{proof}

\begin{remark}
Under the same assumptions for $a $ and $h $ as in Theorem \ref{ABTh03},
we can use \cite[Theorem 2]{F-M-K}
and show the uniform boundedness of $\| x(t) \|$
without any restriction for $p \in [ 1 , \infty ) $
(see also \cite[Theorem 2.2]{I-Y-K} and \cite[Theorem 2.2]{S-I-Y-K}
for global boundedness of solution).
Hence we can derive Theorem \ref{ABTh03} for every $p \in [1 , \infty )$
by referring known abstract results.
We here mention that, however,
the convergence of whole sequence $\{ x(t) \}$ 
without extraction specific subsequence can be obtained  in our result
thanks to   the continuous dependence of solutions (Theorem \ref{Dependence}),
which can not be assured only by \cite{F-M-K, I-Y-K, S-I-Y-K}.
Moreover, we recall that 
the decay estimate \eqref{decay} in Theorem \ref{ABTh01} is derived 
from the  Poincar\'{e} type inequality (Theorem \ref{Poincare}), 
that is, the nature of the hypergraph Laplacian.
\end{remark}

\subsection*{Acknowedgements}
T. Fukao  is supported by 
 JSPS Grant-in-Aid for Scientific Research (C) (No.21K03309). 
M. Ikeda  is supported by JST CREST Grant (No.JPMJCR1913) and JSPS Grant-in-Aid for Young Scientists Research (No.19K14581).
S. Uchida is supported by JSPS Fund for the Promotion of Joint International Research 
(Fostering Joint International Research (B)) (No.18KK0073)
and Sumitomo Foundation Fiscal 2022 Grant for Basic Science Research Projects (No.2200250).


%
\address{
Fukao Takeshi\\
Department of Mathematics, \\
Kyoto University of Education, \\
1 Fujinomori, Fukakusa, \\
Fushimi-ku, Kyoto,  \\
612-8522, JAPAN.%
}
{fukao@kyokyo-u.ac.jp}
%
\address{
Masahiro Ikeda\\
Department of Mathematics, \\
Faculty of Science and Technology, \\   
Keio University, \\
3-14-1 Hiyoshi Kohoku-ku, Yokohama, \\
223-8522, JAPAN/\\
Center for Advanced Intelligence \\
Project, RIKEN, Tokyo, \\
103-0027, JAPAN.
}
{masahiro.ikeda@keio.jp/\\
masahiro.ikeda@riken.jp}
%
%
\address{
Shun Uchida\\
Department of Integrated Science and Technology, \\
Faculty of Science and Technology, \\ 
Oita University,\\
700 Dannoharu, Oita City, Oita Pref., \\
 870-1192, JAPAN.
}
{shunuchida@oita-u.ac.jp}
\end{document}